\documentclass[a4paper,11pt,final]{amsart}

\usepackage{amsmath}
\usepackage{mathtools}
\usepackage[latin1]{inputenc}
\usepackage{varioref}
\usepackage{amsfonts}
\usepackage{amssymb}
\usepackage{bbm}
\usepackage{graphicx}
\usepackage{mathrsfs}
\usepackage[hypertexnames=false,backref=page,pdftex,
 	pdfpagemode=UseNone,
 	breaklinks=true,
 	extension=pdf,
 	colorlinks=true,
 	linkcolor=blue,
 	citecolor=red,
 	urlcolor=blue,
 ]{hyperref}

\usepackage{tikz, tikz-cd}
\usetikzlibrary{arrows}
\usepackage{enumitem}
\setlist[description]{leftmargin=0cm,  labelindent=\parindent}

\newcommand{\sM}{{\mathcal M}}

\newcommand{\sO}{{\mathcal O}}

\newcommand{\scrL}{{\mathscr L}}

\newcommand{\C}{{\mathbb C}}

\renewcommand{\P}{{\mathbb P}}
\newcommand{\Q}{{\mathbb Q}}
\newcommand{\R}{{\mathbb R}}
\newcommand{\IS}{{\mathbb S}}

\newcommand{\Z}{{\mathbb Z}}

\newcommand{\gothF}{{\mathfrak F}}

\newcommand{\gothX}{{\mathfrak X}}

\newcommand{\csphere}{\IS^6}
\newcommand{\dsphere}{S^6}

\newcommand{\abs}[1]{{\left|#1\right|}}

\newcommand{\Aut}{\operatorname{Aut}}

\newcommand{\id}{{\rm id}}
\newcommand{\ch}{{\rm ch}}

\newcommand{\codim}{\operatorname{codim}}

\newcommand{\Exc}{\operatorname{Exc}}

\newcommand{\Hom}{\operatorname{Hom}}

\renewcommand{\implies}{\Rightarrow}

\newcommand{\isom}{{\ \cong\ }}

\newcommand{\ohne}{{\ \setminus \ }}
\newcommand{\ohnenull}{{\ \setminus \{0\} }}

\newcommand{\Pic}{\operatorname{Pic}}

\newcommand{\ratl}{\dashrightarrow}

\newcommand{\rk}{{\rm rk}}

\renewcommand{\to}[1][]{\xrightarrow{\ #1\ }}

\newcommand{\tensor}{\otimes}

\newcommand{\td}{{\rm td}}

\newcommand{\inverse}[1]{{#1}^{-1}}
\newcommand{\Wedge}{\textstyle\bigwedge}

\newcommand{\ke}{{\mathcal E}}
\newcommand{\kf}{{\mathcal F}}

\newcommand{\km}{{\mathcal M}}

\newcommand{\ko}{{\mathcal O}}

\newcommand{\IC}{{\mathbb C}}

\newcommand{\IP}{{\mathbb P}}

\newcommand{\IZ}{{\mathbb Z}}

\newtheoremstyle{citing}
  {}
  {}
  {\itshape}
  {}
  {\bfseries}
  {\textbf{.}}
  {.5em}
  {\thmnote{#3}}

\theoremstyle{plain}
\newtheorem{theorem}{Theorem}
\newtheorem{lemma}[theorem]{Lemma}
\newtheorem{proposition}[theorem]{Proposition}

\newtheorem{corollary}[theorem]{Corollary}

\numberwithin{theorem}{section}

\theoremstyle{definition}
\newtheorem{definition}[theorem]{Definition}

\numberwithin{equation}{section}

\theoremstyle{remark}
\newtheorem{remark}[theorem]{Remark}
\newtheorem*{claim}{Claim}

{\theoremstyle{citing}
}

\newcommand{\Exz}{{\operatorname{Exc}}}

\newcommand{\trdeg}{{\operatorname{trdeg}}}

\theoremstyle{plain}

\theoremstyle{remark}
\newtheorem{example}[theorem]{Example}

\makeatletter
\newcommand*{\da@rightarrow}{\mathchar"0\hexnumber@\symAMSa 4B }
\newcommand*{\da@leftarrow}{\mathchar"0\hexnumber@\symAMSa 4C }
\newcommand*{\xdashrightarrow}[2][]{%
  \mathrel{%
    \mathpalette{\da@xarrow{#1}{#2}{}\da@rightarrow{\,}{}}{}%
  }%
}
\newcommand{\xdashleftarrow}[2][]{%
  \mathrel{%
    \mathpalette{\da@xarrow{#1}{#2}\da@leftarrow{}{}{\,}}{}%
  }%
}
\newcommand*{\da@xarrow}[7]{%
  \sbox0{$\ifx#7\scriptstyle\scriptscriptstyle\else\scriptstyle\fi#5#1#6\m@th$}%
  \sbox2{$\ifx#7\scriptstyle\scriptscriptstyle\else\scriptstyle\fi#5#2#6\m@th$}%
  \sbox4{$#7\dabar@\m@th$}%
  \dimen@=\wd0 %
  \ifdim\wd2 >\dimen@
    \dimen@=\wd2 %
  \fi
  \count@=2 %
  \def\da@bars{\dabar@\dabar@}%
  \@whiledim\count@\wd4<\dimen@\do{%
    \advance\count@\@ne
    \expandafter\def\expandafter\da@bars\expandafter{%
      \da@bars
      \dabar@ 
    }%
  }%
  \mathrel{#3}%
  \mathrel{%
    \mathop{\da@bars}\limits
    \ifx\\#1\\%
    \else
      _{\copy0}%
    \fi
    \ifx\\#2\\%
    \else
      ^{\copy2}%
    \fi
  }%
  \mathrel{#4}%
}
\makeatother

\title[Complex $\csphere$]{The complex geometry of a hypothetical complex structure on $\dsphere$}

\author{Christian Lehn}
\address{Christian Lehn\\ Fakult\"at f\"ur Mathematik\\ Technische Universit\"at Chemnitz\\
Reichenhainer Stra\ss e 39, 09126 Chemnitz, Germany}
\email{christian.lehn@mathematik.tu-chemnitz.de}

\author{S\"onke Rollenske}
 \address{S\"onke Rollenske\\FB 12/Mathematik und Informatik\\
 Philipps-Universit\"at Marburg\\
 Hans-Meerwein-Str. 6\\
 35032 Marburg\\
 Germany}
 \email{rollenske@mathematik.uni-marburg.de}

\author{Caren Schinko}
\address{Caren Schinko\\Lehrstuhl f\"ur Algebra und Zahlentheorie\\
  Universit\"at Augsburg\\Universit\"atsstra{\ss}e 14\\
  86159 Augsburg\\ Germany}
\email{caren.schinko@math.uni-augsburg.de}

\let\origmaketitle\maketitle
\def\maketitle{
  \begingroup
  \def\uppercasenonmath##1{} 
  \let\MakeUppercase\relax 
  \origmaketitle
  \endgroup
}

\begin{document}
\thispagestyle{empty}

\begin{abstract}
This article is a survey about or introduction to certain aspects of the complex geometry of a hypothetical complex structure on the six-sphere. We discuss a result of Peternell--Campana--Demailly on the algebraic dimension of a hypothetical complex six-sphere and give some examples. We also give an overview over an application of Huckleberry--Kebekus--Peternell on the group of biholomorphisms of such a complex six-sphere.
\end{abstract}


\maketitle

\setlength{\parindent}{1em}
\setcounter{tocdepth}{1}
\tableofcontents

\section{Introduction}\label{sec intro}

The Hopf problem concerns the existence of a compact complex manifold whose underlying differentiable manifold is the sphere $\dsphere$; we will refer to this as a hypothetical complex six sphere and denote such a manifold\footnote{Note that there might be more than one $\csphere$.} by $\csphere$. As a compact complex manifold such an $\csphere$ can be studied using the toolbox of complex analytic geometry. The hope, which has not materialized so far, is that one can either derive a contradiction or find a hint for a direct geometric construction thus solving the Hopf problem.

The main results in this direction, besides the restrictions on Hodge numbers described in \cite{Angella}, have been obtained by Campana, Demailly, and Peternell in \cite{CDP} and subsequently by Huckleberry, Kebekus, and Peternell in \cite{HKP}. They concern the algebraic dimension of $\csphere$ and its automorphism group.

Here we do not aim to give a full independent proof of their results, instead we intend this paper as an introduction and stepping stone to their methods and results.

{
We have been notified by Thomas Peternell that the arguments  used in \cite{CDP} to control the algebraic dimension need an additional assumption, and that they are preparing a corrigendum\footnote{After the present article had been published, this corrigendum appeared on the arXiv, see \cite{CDP19}.} completing the study of meromorphic functions on a hypothetical $\csphere$. We were informed that there are three cases to consider, and as in one of them the original argument of \cite{CDP} needs only minor changes we will follow this particular case. We do not claim originality for the material presented here.
}

Thus, we will show in Section \ref{sec main theorem} that a hypothetical $\csphere$ does not admit meromorphic functions that do not extend to a holomorphic map to $\IP^1$. It turns out that this is primarily a consequence of the topology of $S^6$: by the Gau\ss-Bonnet theorem the topological Euler number is given by the integral over the top Chern class of $X$ and due to the vanishing of $b_2(X)$ the Hirzebruch-Riemann-Roch formula shows that $e(X)=0$ is equivalent to $\chi(T_X)=0$ where $\chi$ is the holomorphic Euler characteristic for a complex vector bundle on $X$. This last vanishing is deduced from a more general cohomological vanishing result, Theorem \ref{theorem main2 new}.

Before we describe the method of \cite{HKP} to control the automorphism group of a hypothetical $\csphere$ of algebraic dimension $0$ in Section \ref{sect: aut}, we construct in Section \ref{sec examples}  some examples of non-K\"ahler manifolds to show that Theorem \ref{theorem main2 new} does indeed apply to some manifolds.

We will now start with a gentle introduction to some of the ideas needed to prove the above results.

\subsection*{Acknowledgements}
This research was supported  by the DFG via the first authors research grants Le 3093/2-1 and  Le 3093/3-1 and the second authors Emmy-Noether grant Ro 3734/2-1.

{We are grateful to Thomas Peternell for informing us about the gap in
\cite{CDP} and under which assumptions the approach to the problem is still
viable, in particular, for sending us a preliminary version of \cite{CDP19} where a proof of the case treated here was already contained.}
We would like to thank Giovanni Bazzoni for comments on an earlier version and the participants and organizers of the workshop from which this paper evolved for a fertile week.
The first author would like to thank Patrick Graf for interesting discussions. The third author would like to thank her supervisor Marc Nieper-Wi{\ss}kirchen for his support.

\thispagestyle{empty}

\section{Preliminaries}\label{sec prelim}

\subsection{Meromorphic functions} 
Recall that an analytic subset in a complex manifold is a (closed) subset which locally around every point can be described as zero set of finitely many holomorphic functions. Analytic subsets are by definition the closed subsets of the analytic Zariski topology.

\begin{definition}
Let $X$ be a complex manifold. A \emph{meromorphic function} on an open set $U \subset X$ is a holomorphic function $f\colon U \ohne P \to \C$ where $P \subset U$ is an analytic subset called the \emph{polar set} of $f$ such that for every $p \in U$ there exist an open neighbourhood $V \subset U$ and holomorphic functions $g, h\colon  V \to \C$ such that $h$ has no zeros outside $P$ and $f=\frac{g}{h}$ on $V \ohne P$.
\end{definition}

\begin{remark}\label{rem: ex meromorphic function}
 In higher dimensions, meromorphic functions (and maps), usually denoted by dashed arrows,  behave differently from the familiar case of one complex variable. 
 Consider as an example $X = \IP^2$ with homogeneous coordinates $(x:y:z)$  and the rational map given by $f = x/y\colon X\dashrightarrow \IC$. 
 Then clearly $f$ is not globally a quotient of two holomorphic functions. 
 A maximal domain of definition is $U = X\setminus \{(x: y:z) \mid y=0\}$ so that $\Sigma = \{y=0\}$ becomes the pole divisor of $f$ and $\{x=0\}$ is the divisor of zeros. Note that $f$ cannot be extended to a holomorphic map to $\IP^1$, because it cannot be well defined at the point $(0:0:1)$. Indeed, it is well known, that there is no holomorphic map from $\IP^2$ to $\IP^1$.
 \end{remark}

Meromorphic functions can be added, multiplied, and inverted outside the polar set and thus form a field which leads to the definition of the algebraic dimension

\begin{definition}
Let $X$ be a complex manifold and denote by $\sM(X)$ its field of meromorphic functions. The \emph{algebraic dimension} of $X$ is defined to be $a(X):=\trdeg_\C\sM(X)$.
\end{definition}

The condition $a(X)>0$ is thus equivalent to the existence of a non-constant meromorphic function. Riemann surfaces always satisfy $a(X)=1$. In higher dimension $a(X)$ may attain every value between $0$ and $\dim X$. This notion is not well behaved for non-compact manifolds. It is an interesting exercise to show that $a(\C^N)=\infty$ for example. 
If $X$ is a projective\footnote{By a projective manifold, we understand a closed submanifold of complex projective space $\P^n$. In particular, a projective manifold is always compact. Recall that by Chow's Theorem, a projective manifold is always algebraic.} manifold, then $a(X)=\dim_\C X$. In this case, meromorphic functions on $X$ are the same as rational functions on the corresponding algebraic variety $\gothX$ and it is well known that the field of rational functions determines the algebraic variety $\gothX$ up to bimeromorphic equivalence, see Definition \ref{definition meromorphic map}. 
By a theorem of Thimm \cite[6.11, Hauptsatz III]{Th}, we have $a(X) \leq \dim_\C X$ for a compact complex manifold.\footnote{The result has an interesting history. The first author who published a proof of such a theorem for abelian varieties was Weierstra\ss, see also chapter 4.10 of Fischer's book \cite{Fi} for some historical remarks and a proof of the theorem for reduced and irreducible compact complex spaces. Fischer names the theorem after Weierstra\ss-Siegel-Thimm and Siegel seems to have been the first to give a complete proof of Weierstra\ss' theorem in \cite{Siegel}. Thimm was the first to prove the theorem for arbitrary compact complex manifolds, see  Grauert--Remmert \cite[Chapter 10 \S 6, 4.]{GR} for another proof and some historical remarks. Finally, Thimm himself wrote an article on the history of this theorem, see \cite{Th2}.}

\begin{definition}\label{definition meromorphic map}
Let $X$ and $Y$ be complex manifolds. A \emph{meromorphic map $f\colon X \ratl Y$} is a holomorphic map $f\colon X\ohne P \to Y$ where $P \subset X$ is a nowhere dense closed subset with the following properties: there is an analytic subset $\Gamma \subset X \times Y$ such that the map $p\colon \Gamma \to X$ induced by the projection to the first factor is proper, is biholomorphic over $X\ohne P$, the subset $p^{-1}(P)  \subset \Gamma$ is nowhere dense, and $f=q \circ p^{-1}$ where $q\colon \Gamma \to Y$ is induced by projection to the second factor. The set $\Gamma$ is uniquely determined by $f$ and is called the \emph{graph of $f$}. The subset $X\ohne P$ is called the \emph{domain of definition} of $f$ and its complement $P$ is called the \emph{locus of indeterminacy} of $f$. We usually consider two meromorphic maps $f,f'\colon X \ratl Y$ equal if they coincide on the intersection of their domain of definitions. Thus, there is an obvious notion of composition for meromorphic maps which are dominant\footnote{Recall that a holomorphic map is \emph{dominant} if its image contains a dense open subset. }.
A meromorphic map $f\colon X \ratl  Y$ is called \emph{bimeromorphic} if there exists a meromorphic map $g\colon Y \ratl X$ such that $f\circ g = \id_Y$ and $g\circ f=\id_X$. 
\end{definition}

A meromorphic function $f\colon X \ratl \C$ thus gives rise to a meromorphic map $f\colon X \ratl \P^1$. A \emph{fibre} of a meromorphic map $f\colon  X \ratl Y$ is the closure of a fibre of the restriction of $f$ to its domain of definition. So in particular, fibres over different points of $Y$ may intersect in $X$. 

It is convenient to introduce the following terminology.

\begin{definition}
A holomorphic map $f\colon Y \to X$ between complex manifolds is a \emph{proper modification} if $f$ is proper and bimeromorphic.
\end{definition}

\begin{definition}
Let $g\colon X \ratl Z$ be a meromorphic map.  A \emph{resolution of indeterminacy} for $g$ is a commuting diagram 
\[
\begin{tikzcd}
 Y \dar[swap]{\pi} \arrow{dr}{f}\\
 X \rar[dashed][swap]{g} & Z
\end{tikzcd}
\]
where $\pi \colon  Y \to X$ is a proper modification and $f$ is holomorphic. The diagram is said to commute if it commutes on a dense open subset where all maps are well-defined.
\end{definition}

Resolutions of indeterminacy where $Y$ is a complex manifold always exist: just take a resolution of singularities of the graph $\Gamma \subset X \times Z$. As we observed, a meromorphic function on $X$ gives a meromorphic map $f\colon X \ratl \P^1$. One can always obtain from this a meromorphic map $g\colon X \ratl C$ with connected fibres to a smooth compact curve $C$ by taking the Stein factorization of a resolution of indeterminacy.

\begin{remark}
 Going back to the example from Remark \ref{rem: ex meromorphic function} we consider the meromorphic map $f = x/y \colon \IP^2 \dashrightarrow \IP^1$. The locus of indeterminacy is the point $P = (0:0:1)$ and the closure of the graph of $f|_{X\setminus\{P\}}$ in the product $\IP^2 \times \IP^1$ is the blow up of $\IP^2$ in the point $P$, which is the $\IP^1$-bundle $\IP(\ko\oplus\ko(-1))\to \IP^1$. The fibres of $f$ are the images of the fibres of this $\IP^1$-bundle, that is, the lines through the point $P$. 
 This is the classical linear projection from a point. It is an easy but instructive exercise in classical geometry to understand the projection from two distinct points, that is, the rational map 
 \[ \IP^2 \ratl \IP^1\times \IP^1, \qquad (x: y: z) \mapsto \left(\frac x y , \frac z y \right).\]
\end{remark}

\begin{lemma}\label{lemma pullback}
Let $\pi\colon Y \to X$ be a proper modification from an integral (i.e., irreducible and reduced) compact complex space to a compact complex manifold and let $D$ be an effective Cartier divisor on $Y$. If $D$ does not contain any divisorial component of the  $\pi$-exceptional locus, then there exists a unique line bundle $L$ on $X$ and an effective Cartier divisor $\Sigma \subset Y$ whose support is contained in the exceptional locus of $\pi$ 
such that
\[\ko_Y(D) = \pi^*L \tensor \sO_Y(-\Sigma).\]
\end{lemma}
\begin{proof}
Consider the divisor $\bar D = \pi_*D$, that is, the 1-codimensional part of the image of $D$. Since $\bar D$ is effective, its pullback to $Y$ is also effective and thus there is an effective $\pi$-exceptional divisor $\Sigma$ such that 
\[ \pi^*\bar D = \pi^*\pi_*D = D+\Sigma,\]
because $D$ does not contain components of $\Exc(\pi)$.
Setting $L = \ko_X(\bar D)$ our claim follows.
\end{proof}

Having no non-constant meromorphic functions has some consequences which will turn out to be useful later:
\begin{proposition}\label{prop: wedging sections}
 Let $X$ be a compact complex manifold of algebraic dimension $a(X) =0$ and let $\ke$ be a holomorphic vector bundle of rank $r$ on $X$. If $s_1, \dots, s_k\in H^0(X, \ke)$ are linearly independent (over $\IC$)  then $s_1\wedge \dots \wedge s_k$ is a non-zero holomorphic section of $\Wedge^k\ke$.

 In particular, 
 \[ h^0\left(X, \Wedge^k \ke\right)\geq k\left( h^0(X, \ke)-k\right)+1\]
 and thus $h^0(X, \ke)\leq \rk(\ke)$. 
 \end{proposition}
\begin{proof}
 By induction we may assume that for all $j=1, \dots, k$ the wedge products $s_1\wedge \dots \wedge \widehat s_j \wedge \dots \wedge s_k\neq 0$, that is, all subsets of $k-1$ of the sections are linearly independent in the fibre at the general point of $X$. 
 
 Assume $s_1\wedge \dots \wedge s_k=0$ in $H^0\left(X, \Wedge^k \ke\right)$.
 Then in every local trivialization $\ke|_U\isom \ko_U^{\oplus r}$ over a connected open subset $U\subset X$ we can solve the system of linear equations over the field $\km(U)$ of meromorphic functions on $U$ to get a unique linear relation
 \[s_k = \sum _{i=1}^{k-1} \lambda_i^U s_i\]
 with $\lambda^U_i\in \km(U)$. 
 By uniqueness these local linear relations extend to a global linear relation over the field $\km(X)$. However, since $a(X) =0$ we have $\km(X) = \IC$ and thus  the $s_i$ are not $\IC$-linearly independent --- a contradiction.
 
To get the dimension estimate we interpret decomposable elements of the form $s_1\wedge \dots\wedge s_k\in \Wedge ^k H^0(X, \ke)$ as representatives of elements in the Grassmannian $\mathrm{Gr}(k, h^0(X, \ke))$ of $k$-planes in $H^0(X, \ke)$ in its Pl\"ucker embedding. 
Then the above says that the indeterminacy locus of the linear projection 
\[\IP\left( \Wedge^k H^0(X, \ke)\right) \dashrightarrow \IP\left(H^0\left(X, \Wedge^k \ke\right)\right)\]
does not intersect $\mathrm{Gr}(k, h^0(X, \ke))$, so that in particular 
\[ h^0\left(X, \Wedge^k \ke\right)-1\geq \dim  \mathrm{Gr}(k, h^0(X, \ke)) = k\left(( h^0(X, \ke)-k\right).\]
Substituting $k = \rk(\ke)+1$ gives the last estimate. 
\end{proof}

\subsection{Sheaf cohomology and higher direct images}\label{sec: sheaf cohomology}

The obstruction to globalise a local construction on a compact complex manifold $X$  can usually be found in the appropriate sheaf cohomology group. For an introduction to these important techniques we refer to \cite{Ha}, but we will briefly discuss higher direct image sheaves, which will make a prominent appearance later on.
Let $f\colon Y\to X$ be a proper holomorphic map and $\kf$ a sheaf of abelian groups on $Y$. Then we define the $i$-th higher direct image sheaf $R^if_*\kf$ to be the sheaf on $X$ associated to the presheaf
\[ U\mapsto H^i(\inverse f U, \kf).\]
Even though this is not always true, one would like to imagine this to be the sheaf on $X$, whose fibre at $x$ is the cohomology of $\kf$ restricted to $\inverse f(x)$ and this can be made precise by the base change theorem \cite{BS}. In addition, these sheaves measure the failure of $f_*$ to be exact in the sense that if 
\[ 0\to \kf' \to \kf \to \kf''\to 0\]
is an exact sequence of sheaves on $Y$ then we get a long exact sequence
\[ 0 \to f_*\kf'\to f_*\kf \to f_*\kf''\to R^1f_*\kf'\to R^1f_*\kf \to \dots\]
of sheaves on $X$. 
This tool is often used to compute cohomology groups because of the Leray spectral sequence
\[ E_2^{p,q} = H^q(X, R^pf_*\kf) \implies H^{p+q}(Y, \kf).\]
As an example, one might want to consider the sheaf $\IZ_Y$ on the total space of a fibre bundle and reprove the Leray-Hirsch Theorem with this machinery. 

\begin{remark}\label{remark coherent sheaves}
 A particularly important class of sheaves on a complex manifold $X$ are coherent sheaves, which are roughly all sheaves that can be written as cokernels of an $\sO_X$-linear map of locally free sheaves. 
 These have the useful property that they have no cohomology in degree higher than the dimension of $X$, which implies that higher pushforward sheaves are zero above the maximal dimension of a fibre. 
\end{remark}

\subsection{The Riemann-Roch theorem} 
We recall from the theory of characteristic classes (see for example \cite[Appendix A.4]{Ha} or \cite{maurizio}) the following expressions for the Todd class and the Chern character of a complex vector bundle: 
\begin{equation}\label{eq chern classes}
\begin{aligned}
 \td  & = 1 + \frac{c_1}{2}+\frac{c_1^2+c_2}{12}+ \frac{c_1c_2}{24}+\ldots \\
 \ch & = \rk E+c_1 + \frac{c_1^2-2c_2}{2} + \frac{c_1^3-3c_1c_2+3c_3}{6} + \ldots .  
\end{aligned}
\end{equation}

What follows is rather standard for smooth projective varieties, but we did not find a comprehensive treatment for compact complex manifolds in the literature which is why we wanted to provide some references.
An important tool is the Hirzebruch-Riemann-Roch theorem which states that for a vector bundle or a coherent sheaf $F$ on a compact complex manifold $X$ of dimension $n$ the holomorphic Euler characteristic $\chi(F):=\sum_{i=0}^n\dim H^i(X,F)$ can be expressed in terms of the Chern character of $F$ and the Todd class of $X$ as follows:
\begin{equation}\label{eq hrr formula}
\chi(F)=\int_X\ch(F)\td(X).
\end{equation}
The formula has first been proven for a smooth projective variety $X$. It has been proven in the analytic context first for characteristic classes in the Hodge cohomology ring $\bigoplus_{p=0}^nH^p(X,\Omega_X^p)$ by \cite{OBTT}. A version in K-theory for complex spaces has been proven in \cite{L} from which the singular cohomology version can be deduced as for smooth projective varieties, see \cite{BFMP2}.

The Hirzebruch-Riemann-Roch formula is used to connect holomorphic information with topological information. This is also manifested in the Gau\ss-Bonnet theorem, which links the topological Euler characteristic of a compact complex manifold $X$ to its top Chern class.
\begin{equation}\label{eq gauss bonnet}
e(X)=\int_Xc_n(X).
\end{equation}

It is well known how to deduce \eqref{eq gauss bonnet} from Hirzebruch-Riemann-Roch, see for example Huybrechts' book on complex geometry \cite[Corollary 5.1.4]{Hu}. We obtain
\[
\int_X c_n(X) = \sum_{p= 0}^n (-1)^{p}\chi(\Omega_X^p) = \sum_{i=0}^n (-1)^i b_i(X) = e(X).
\]
The first equality can also be obtained directly from the Borel-Serre formula, see \cite[Example 3.2.5]{fulton}, and the Hirzebruch-Riemann-Roch theorem. Then we use that the holomorphic de Rham complex is quasi-isomorphic to the constant sheaf $\C$. Note that it is not necessary that the spectral sequence $E_1^{pq}=H^q(X,\Omega_X^p) \implies H^{p+q}(X,\C)$ degenerates at $E_1$. The Euler characteristic of the $E_1$-term of a spectral sequence is always equal to the Euler characteristic of the abutment. We observe that by compactness of $X$ all vector spaces are finite dimensional.

\subsection{Picard variety} \label{subsection picard variety}

Let us recall that on a complex manifold $X$  a holomorphic line bundle can be described by a holomorphic cocycle and that up to isomorphism holomorphic line bundles are classified by the (\v{C}ech)-cohomology group $\Pic(X) = H^1(X, \ko_X^\times)$. The exponential sequence of sheaves
\[
0\to \Z_X \to \sO_X \to[\exp(2\pi i \cdot ) ] \sO_X^\times \to 0
\]
induces an exact sequence
\[
H^1(X,\Z) \to H^1(X,\sO_X) \to \Pic(X)=H^1(X,\sO_X^\times) \to[c_1] H^2(X,\Z) \to \ldots,
\]
where $c_1$ maps the isomorphism class of a line bundle to its first Chern class. 
The Picard variety of degree $0$-line bundles on $X$ is thus $\Pic^0(X)=\ker c_1=H^1(X,\sO_X)/H^1(X,\Z).$ To endow $\Pic^0(X)$ with the structure of a complex manifold we need the following Lemma, for which we give a proof since we could not track down a reference in the literature.

\begin{lemma}\label{lemma pic zero}
For every compact complex manifold the group $H^1(X,\Z)$ is discrete in $H^1(X,\sO_X)$. In particular, $\Pic^0(X)$ is a connected Hausdorff complex manifold.
\end{lemma}
\begin{proof}
We consider the short exact sequence $0 \to \C \to \sO_X \to d\sO_X \to 0$ of sheaves on $X$ where $d\sO_X \subset \Omega_X$ is the subsheaf of locally exact holomorphic $1$-forms. We see from this that 
$$ H^1(X,\C)/H^1(X,d\sO_X) \subset H^1(\sO_X).$$
We have $H^1(X,d\sO_X) \cap \overline{H^1(X,d\sO_X) } =0$ by Lemma 2.2 of \cite{Catanese}. Therefore, $H^1(X,\R)$ embeds into $H^1(X,\C)/H^1(X,d\sO_X)$ and thus $H^1(X,\Z)$ is discrete in $H^1(X,\sO_X)$.
Finally, $\Pic^0(X)$ is connected as it is the image of $H^1(X,\sO_X)$.
\end{proof}

For a compact K\"ahler manifold $X$, Hodge theory tells us that $\Pic^0(X)$ is compact, as $b_1(X) = 2h^1(X, \ko_X)$. This need no longer be the case for compact complex manifolds: if for example the first Betti number is odd, the $H^1(X, \IZ)$ can never  map to a lattice of full rank in $H^1(X, \ko_X)$. We give a classical example where this occurs.

\begin{example}\label{ex: Hopf}
Let $X$ be a Hopf surface defined as $X= \left(\C^2\ohnenull\right) / G$ where $G$ is the infinite cyclic group generated by a contraction $\gamma$. Kodaira has observed in section 10 of \cite{kodaira123} that by work of Latt\`es \cite{lattes} and Sternberg \cite{sternberg} such an automorphism $\gamma$ of $\C^2\ohnenull$ can be brought into a normal form given by $(z_1,z_2) \mapsto (\alpha_1 z_1 + \lambda z_2^n,\alpha_2 z_2)$ where $\alpha_1, \alpha_2, \lambda \in \C$ with $0 < \abs{\alpha_1}\leq \abs{\alpha_2} < 1$ and $\lambda =0$ unless $\alpha_1 = \alpha_2^n$. One easily sees that a Hopf surface $X$ is diffeomorphic to $\IS^1\times \IS^3$ so that $H^1(X,\Z)=\Z$ and $H^2(X,\Z)=0$. In fact, Kodaira \cite{kodaira} has shown that all complex structures on $\IS^1\times \IS^3$ come from Hopf surfaces.
It can also be shown that $H^1(X,\sO_X)=\C$, see \cite[V.18]{BHPV}. The exponential sequence therefore gives us an exact sequence 
\[
0 \to H^1(X,\Z) \to H^1(X,\sO_X) \to H^1(X,\sO_X^\times) \to 0
\]
so that $\Pic^0(X) = \Pic(X) \isom \C^\times$.
\end{example}

\begin{remark}\label{remark universal line bundle}
From the exponential sequence it is straightforward to construct a universal line bundle $\scrL$ on $\Pic^0(X) \times X$ with the property that for all isomorphism classes of line bundles $[L]\in \Pic^0(X)$ we have $\scrL\vert_{\{[L]\}\times X} \isom L$ where we identified $\{[L]\}\times X$ with $X$ via the second projection.

Let $V=H^1(X,\sO_X)$, let $\pi\colon V \times X \to X$ be the second projection and consider the tautological section 
\[
\tau \in H^0(V \times X,\ko_{V \times X}\tensor H^1(X,\ko_X) )  = H^0(V,\ko_{V})\tensor H^1(X,\ko_X)
\]
given by $\tau(\alpha, x) = \alpha$. If we consider $\tau$ as an element of $H^1(V \times X,\ko_{V \times X})$ via the K\"unneth decomposition and then apply the exponential map we get a class $\exp(2\pi i\tau)\in H^1(V \times X, \ko_{V \times X}^\times)$ which defines a line bundle on $V \times X$. By construction, this line bundle descends to the universal line bundle $\scrL$ on $\Pic^0(X)\times X$. 
\end{remark}

\section{The algebraic dimension of a hypothetical complex \texorpdfstring{$\csphere$}{S6}}\label{sec main theorem}
{ 
In \cite{CDP} Campana, Demailly, and Peternell proved, as an application of a more general result, that a hypothetical complex $\csphere$ has no meromorphic functions, that is, has algebraic dimension zero. 
However, we were notified by Thomas Peternell, that the proof has a gap. We will explain the gap  in Remark \ref{rem: gap}. 
He also notified us, that a corrigendum is in preparation and,  while the more general result has to be weakened, the conclusion for the hypothetical complex $\csphere$ still stands. 
}

{In this section, our goal is more modest: we will give a general vanishing result for threefolds with $b_2=0$ admitting non-holomorphic, meromorphic maps to a curve following the original strategy of \cite{CDP} and explanations by Thomas Peternell. This will exclude the existence of such functions on a hypothetical complex $\csphere$.}

\begin{theorem}\label{theorem main2 new}
Let $X$ be a compact complex manifold with $\dim_\C X=3$, $b_2(X)=0$, and assume that $X$ admits a meromorphic, non-holomorphic map  $f\colon X\dashrightarrow C$  onto a curve. Then for a holomorphic vector bundle $E$ on $X$ we have
\begin{enumerate}
	\item $H^i(X,E\tensor L) = 0$ for all $i \geq 0$ and for a holomorphic line bundle $L$ on $X$ which is generic\footnote{In this context, \emph{generic} means that the statement applies to $L$ contained in the complement of some nowhere dense analytic subset of $\Pic^0(X)$.}   in $\Pic^0(X)$.
	\item $\chi(E\tensor L) = 0$ for every holomorphic line bundle $L$ on $X$ in $\Pic^0(X)$.
	\item $c_3(X)=0$.
\end{enumerate}
\end{theorem}
The Gau\ss-Bonnet formula \eqref{eq gauss bonnet} implies immediately.
\begin{corollary}[Campana--Demailly--Peternell]\label{cor: partial a(S6)=0}
Let $\csphere$ be a compact complex manifold whose underlying topological space is homeomorphic to the $6$-sphere $\dsphere$, then $X$ does not admit  a meromorphic, non-holomorphic map  $f\colon X\dashrightarrow C$  onto a curve.

In particular, if $a(X)>0$, then every meromorphic function on $X$ extends to a holomorphic map to $\IP^1$. 
\end{corollary}

For the proof of Theorem \ref{theorem main2 new} we need an auxiliary result. 

\begin{lemma}\label{lemma higher direct images}
Let $X$ be a smooth complex threefold, let  $Y \subset X \times \P^1$ be a closed subspace which is integral,\footnote{irreducible and reduced} and suppose that the projection $\pi\colon  Y \to X$ is bimeromorphic. Then the following holds.
\begin{enumerate}
	\item Let $F$ be a coherent sheaf on $Y$. Then $R^q\pi_* F =0$ for all $q\geq 2$.
	\item $R^q\pi_*\sO_Y = 0$ for $q>0$ und $\pi_*\sO_Y=\sO_X$.
	\item Let $\Sigma \subset Y$ be a closed subspace (not necessarily reduced). Then $R^1 \pi_*\sO_\Sigma=0$.
\end{enumerate} 
\end{lemma}

\begin{proof}
Ad (1): 
As the fibres of $\pi$ are contained in $\P^1$, the $R^2\pi_*\ko_\Sigma$ vanishes for dimensional reasons. This is a consequence of Grauert's comparison theorem\footnote{In algebraic geometry this statement is known as the theorem on formal functions.} \cite[\S 6, Hauptsatz II a]{grauert}, see also \cite[II. Theorem 4.5 (1)]{SCV7}.

Ad (2): The statement for $q=2$ is a special case of (1). As the exceptional fibres of $\pi$ are all isomorphic to $\P^1$, the claim for $q=1$ follows as in (1) using that $H^1(\P^1,\sO_{\P^1})=0$. The statement $\pi_*\sO_Y = \sO_X$ is a consequence of the existence of a Stein factorization, birationality of $\pi$, and normality of $X$. Finally, (3) is obtained from (1) and (2) by invoking the long exact sequence of higher direct images associated to $0 \to \sO_Y(-\Sigma) \to \sO_Y \to \sO_\Sigma \to 0$.
\end{proof}

The following two examples show that for birational morphisms $\pi: Y \to X$ between compact complex threefolds in general neither $\pi_*\ko_\Sigma  = \ko_{\pi(\Sigma)}$ nor $R^1\pi_*\ko_\Sigma =0$. 

\begin{example}
\begin{enumerate}
	\item Let $X$ be a compact complex threefold containing a smooth curve $C \subset X$. Let $\pi_1\colon X_1 \to X$ be the blow up of $X$ along $C$ and let $\pi_2\colon X_2 \to X_1$ be the blow up of $X_1$ along a smooth curve $C'$ contained in $\Exz(\pi_1)$ such that the induced map $p\colon C' \to C$ is finite of degree $d>1$. Such a curve clearly exists: as $\Exz(\pi_1)$ is projective, any general ample divisor will do. Let $\Sigma=\Exz(\pi_2)$ and let $\pi\colon  \Sigma \to C$ be the composition of $\pi_1 \circ \pi_2$ restricted to $\Sigma$. Then $\pi_* \sO_\Sigma = p_*\sO_C'$ which is a coherent sheaf of rank $d$.
	\item Let $X$ be a compact complex threefold, let $\pi_1\colon X_1\to X$ be the blow up of a smooth point  $p$
with exceptional divisor $E$ and let $\pi_2\colon X_2\to X_1$ the blow up of a
smooth non-rational curve $C$ contained in $E$ with exceptional divisor
$\Sigma$.
Then $R\pi_{2*}\ko_\Sigma = \ko_C$ and
$R(\pi_{1}\circ \pi_{2})_*\ko_\Sigma = R\pi_{1*}\ko_C$ so that $\left(R^1(\pi_1\circ\pi_2)_*\sO_\Sigma\right)_p = H^1(C,\sO_C)\neq 0$.
\end{enumerate}
\end{example}

Every non-zero  holomorphic section of a vector bundle on a manifold can vanishes along a given divisor only to finite order. Based on this idea one gets a vanishing criterion in a more general situation.
\begin{lemma}[\protect{\cite[Prop.\ 1.1]{CDP}}]\label{lem: vanishing by twisting}
 Let $Z$ be a compact complex space, not necessarily reduced, but without embedded components,  and let $D\subset Z$ be an effective Cartier divisor. Assume that $D$ meets every irreducible component of $Z$ non-trivially. 
Then for every holomorphic vector bundle $E$ on $X$ there exists $k_0$ such that 
 $H^0(Z, E(-kD)) = 0$ for $k\geq k_0$.
\end{lemma}
The following observation is rather trivial but crucial for the proof.
\begin{lemma}\label{lem: power in Pic0}
 Let $X$ be a complex manifold with $b_2(X) = 0$ and let $L\in \Pic(X)$.
 Then for $m$ sufficiently divisible $mL\in \Pic^0(X)$. 
  \end{lemma}
\begin{proof}
Since $b_2(X) = 0$ the integral second cohomology is a finite abelian group and thus the exponential sequence (see Section \ref{subsection picard variety})  induces a short exact sequence
 \[ 0\to \Pic^0(X) \to \Pic(X) \to H^2(X, \IZ)\to 0.\]
 Choosing $m$ to be a multiple of  the exponent of $H^2(X, \IZ) $ proves the claim.
\end{proof}

\begin{proof}[Proof of Theorem \ref{theorem main2 new}]
We will show first that (1)$\implies$(2)$\implies$(3) using only the hypothesis $b_2(X) = 0$. To deduce item (2) from (1) consider the variety $\Pic^0(X)\times X$, the projection $p:\Pic^0(X) \times X \to X$, and the universal line bundle $\scrL$ explained in Remark \ref{remark universal line bundle}. Then we can view the vector bundle $p^*E\tensor \scrL$ as a family of vector bundles on $X$ parametrized by $\Pic^0(X)$. Since $\Pic^0(X)$ is connected (and the family is automatically flat) the holomorphic Euler characteristic $\chi(E\tensor L)$  is independent of $L\in \Pic^0(X)$, see \cite[III. Theorem 4.12 (iii)]{BS}.  From (1) we deduce that $\chi(E\tensor L)=0$ as claimed.

For (3) observe that the vanishing $H^2(X,\Q)=0=H^4(X,\Q)$ implies $c_1(E)=0=c_2(E)$ for every vector bundle $E$ on $X$ where the Chern classes are considered in $H^*(X,\Q)$. In particular, it follows from (2) and \eqref{eq hrr formula} that $0=\chi(T_X) = \frac 12\int_X c_3(X)$, see also \eqref{eq chern classes}.

It remains to show (1). Fix the vector bundle $E$ under consideration. 

\begin{description}
\item[Step 1] By upper semi-continuity of $\dim H^i(X,E \tensor L)$ in families of sheaves (see \cite[III. Theorem 4.12 (i)]{BS}) we obtain that $H^i(X,E \tensor L)=0$ for all $L$ from a dense (Zariski) open subset $U_i \subset \Pic^0(X)$ as soon as we find one such $L$. As the intersection of finitely many dense open subsets is again dense open, it suffices to find for each $i=0,\ldots,3$ a single $L=L_i \in \Pic^0(X)$ for which $H^i(X,E \tensor L)=0$. 
By Serre duality for compact complex spaces (see \cite{AK}) it suffices to show the claim in the cases $i=0$ and $i=2$.

 \item[Step 2] Let us show that $H^0(X, E\tensor L) = 0$ for some $L\in \Pic^0(X)$. For this part we only need to assume that $X$ contains an effective divisor $D$. Since we assumed the existence of a meromorphic function, we can choose $D$ to be one of the fibres.
  Then by Lemma \ref{lem: power in Pic0} there exists an $m$ such that $L = \ko_X(mD)\in \Pic^0(X)$ and $H^0(X, E\tensor L^{-k}) =0$ for $k\gg 0$ by Lemma \ref{lem: vanishing by twisting}. This settles the case of $H^0$.

 \item[Step 3] Let us show that $H^2(X, E\tensor L) = 0$ for some $L\in \Pic^0(X)$.
This vanishing is much more difficult to obtain and we will heavily use the meromorphic map~$f$. Resolving the indeterminacies of $f$ we get a diagram
 \[ \begin{tikzcd}
      {} & Y\arrow{dl}[swap]{\pi} \arrow{dr}{g}\\
      X \arrow[dashed]{rr}{f}& & C
    \end{tikzcd},
\]
where $Y$ is the closure of the graph of $f$ in $X \times C$ and $\pi$ and $g$ are induced by the projections. By assumption, $\pi$ is not an isomorphism. As $X$ is normal and $\pi$ is birational, we have $\pi_*\sO_Y=\sO_X$; in particular, $\pi$ has connected fibres. Replacing $g$ by its Stein factorization we may assume that also $g_*\sO_Y=\sO_C$. Note that $Y$ may fail to be smooth.

\item[Step 3a] We will show first that $C=\P^1$. Let $Z \to Y$ be a resolution of singularities such that the composition $Z \to X$ is a blowup in smooth centres: $Z=Z_n \to[\pi_n] Z_{n-1} \to[\pi_{n-1}] \ldots \to[\pi_{1}] Z_0=X$. Assume that the genus of $C$ is $g(C)>0$. We may assume that $n$ is minimal such that $Z_n \to C$ is holomorphic. If $Z_n$ is the blowup of $Z_{n-1}$ in a smooth subvariety $P$, then the exceptional fibres of $\pi_n$ are isomorphic to $\P^{\codim P -1}$, hence are contracted under $Z_n \to C$ which therefore factors through $Z_{n-1}$ contradicting the minimality of $n$.

\item[Step 3b] Choice of $L$. 
We fix an ample line bundle $A$ on $C=\P^1$. Then by Lemma \ref{lemma pullback} there is a line bundle $L_0$ on $X$ and an effective divisor $\Sigma \subset Y$ whose support is contained in the exceptional locus $\Exc(\pi)$ such that
\begin{equation}\label{eq geradenbuendel}
g^*A = \pi^*L_0 \tensor \sO_{Y}(-\Sigma).
\end{equation}
Note that by our assumption that $f$ does not extend to a holomorphic map from $X$ to $C$, every component of the exceptional locus (in particular every component of $\Sigma$) surjects onto $C$. 
By Lemma~\ref{lem: power in Pic0} we may replace $A$ by a multiple and then assume that $L_0\in \Pic^0(X)$. 
As $Y$ is irreducible and $g_*\sO_Y = \sO_C$, the general fibre of $g$ is  irreducible. 

Since $Y$ was defined to be the graph of $f$, the fibres of $g$ are embedded in $X$ and there is (at most) a finite number of reducible fibres. 
Let us denote $F^\Sigma$ the divisor on $X$ which is the union of all components of reducible fibres of $g$ which meet $\Sigma$ (on $Y$). Observe that this is a Cartier divisor as $X$ is smooth. Let $\gothF$ denote the set of all components of reducible fibres of $g$.
Now we introduce a weight for elements of $\gothF$.
Consider the graph whose vertices are elements of $\gothF$ and where two vertices are joined by an edge if and only if the corresponding components have nonempty intersection in $X$. Observe that as $g$ is not holomorphic there may be edges between vertices corresponding to components of different fibres of $g$.
Define the weight of an $F\in\gothF$ to be the minimal length of a path from the corresponding vertex to a vertex representing a component of $F^\Sigma$. As every fibre is connected and meets $\Sigma$ there is always such a path.
Let us denote by $F_i$ the divisor on $X$ which is the sum of all $F\in\gothF$ that have weight $i$ and let $d+1$ denote the maximal weight that shows up. Note that $F^\Sigma= F_0$ and the support of $F_0 + \ldots + F_{d+1}$ is the union of all reducible fibres of $g$ considered as a subvariety in $X$.
Now consider the line bundle $M$ on $X$ given by
\[ 
M=\sO_X(n_0 F_0 + n_1 F_1 + \ldots + n_d F_d).
\]
Multiplying by a suitable positive integer we may assume $M \in \Pic^0(X)$ thanks to Lemma \ref{lem: power in Pic0}.
We will need the following statement later in the proof.

\begin{claim}
Let $V$ be a locally free sheaf on $Y$ and let $0 \ll n_d \ll \ldots \ll n_0 \ll~k$ be integers. Then for every component $F$ of a fibre of $g$ the equality
\begin{equation}\label{eq m vanishing}
H^0(F,V \tensor \pi^*(M^\vee) \tensor \sO_{Y}(-k\Sigma)\vert_F)=0
\end{equation}
holds. 
\end{claim}
\begin{proof}[Proof of Claim]
As $\iota_X = \pi \circ \iota_Y$ where $\iota_X,\iota_Y$ are the inclusions of $F$ to $X$ and $Y$, we have to show that 
\[
H^0(F,\iota_X^*(V \tensor M^\vee) \tensor \iota_Y^*\sO_{Y}(-k\Sigma))=0.
\]

First, we show that if  $n_d$ is big enough we obtain the claim for all $F$ of weight $d+1$. For this we simply observe that $\Sigma$ and $F_i$ for all $i$ restrict to an effective divisor on $F$ (note that $F$ is not contained in $\Sigma$ because every exceptional divisor surjects onto $C$) and $F_d$ restricts to a non-zero effective divisor on $F$ so that the claim follows from Lemma \ref{lem: vanishing by twisting}. Suppose that $F$ has weight $i>0$ and that the integers $n_d, n_{d-1},\ldots,n_i$ are already chosen to obtain the claim for components of weight $\geq i+1$. Then $\sO_F(n_i F_i + n_1 F_1 + \ldots + n_d F_d)$ is some line bundle, the $F_j$, $j\leq i-1$, and $\Sigma$ restrict to an effective divisor on $F$ and $F_{i-1}$ restricts to a non-zero effective divisor on $F$ so that the claim follows from Lemma \ref{lem: vanishing by twisting} if we choose $n_{i-1}$ big enough.\footnote{Indeed, it is maybe not necessary the $n_i$ and $k$ are ordered in the way formulated in the Claim, but like this we are sure that we can later on (namely in \eqref{eq leray vanishing} for $p=1$) still increase $k$ and this does not affect the validity of the Claim.} Finally, if $F$ has weight zero, the claim follows by adjusting $k$.
\end{proof}

Continuing the proof of Theorem \ref{theorem main2 new}, we choose $L =  L_0^{m+k} \tensor M$ with $m,k,n_0, \ldots, n_d$ still to be determined. 
With this choice, \eqref{eq geradenbuendel} yields
\begin{equation}\label{eq geradenbuendel2}
\pi^*L = \pi^*\left(L_0^{\tensor m} \tensor M \right)\tensor g^*A^k\tensor \sO_{Y}(k\Sigma).
\end{equation}

\item[Step 3c] 
We claim that for suitably chosen $m,k,n_0, \ldots, n_d > 0$ we have $H^2(X,E\tensor L)=0$.
 For this  it is sufficient to show that for some $k \gg 0$ we have $H^2(Y,\pi^*(E\tensor L)(-k\Sigma))=0$ as we show next. Let us abbreviate $\Sigma_k:=k\Sigma$. The last claim follows if we show that the canonical morphisms
\[
\begin{tikzcd}
 H^2(X,E\tensor L) \rar[right hook->]  & H^2(Y,\pi^*(E\tensor L))\\ & H^2(Y,\pi^*(E\tensor L)(-\Sigma_k)) \arrow[twoheadrightarrow]{u}
\end{tikzcd}
\]
are injective respectively surjective. Injectivity follows from the Leray spectral sequence for $\pi$, which collapses at $E_2$ by the projection formula and the fact that $R\pi_*\sO_{Y}=\sO_X$ by (3) of Lemma \ref{lemma higher direct images}. To deduce surjectivity, let $B_k \subset X$ be the closed subspace defined by the sheaf of ideals $\pi_*\sO_{Y}(-\Sigma_k) \subset \sO_{Y}$ and look at the short exact sequence 
\[
0 \to \pi^*(E\tensor L)(-\Sigma_k) \to \pi^*(E\tensor L) \to \pi^*(E\tensor L)\vert_{\Sigma_k} \to 0.
\]
The induced long exact sequence of cohomology groups
\[
\to H^2(Y,\pi^*(E\tensor L)(-\Sigma_k)) \to H^2(Y,\pi^*(E\tensor L)) \to H^2(\Sigma_k,\pi^*(E\tensor L)\vert_{\Sigma_k}) 
\]
allows to deduce surjectivity from the vanishing 
$$H^2(\Sigma_k,\pi^*(E\tensor L)\vert_{\Sigma_k})=H^2(B_k,(E\tensor L)\vert_{\Sigma_k}\tensor \pi_*\sO_{\Sigma_k})=0$$ where for the first equality we invoke Lemma \ref{lemma higher direct images} and the second equality follows from $\dim B_k\leq 1$. 
	To calculate $$H^2(Y,\pi^*(E\tensor L)(-\Sigma_k)) = H^2(Y,\pi^*(E\tensor L_0^{\tensor m}) \tensor g^*A^{\tensor k})$$ via the Leray spectral sequence we have to show the vanishing of
\begin{equation}\label{eq leray vanishing}
H^p\left(C, R^q g_* \left(\pi^*(E\tensor L_0^{\tensor m}\tensor M)\right) \tensor A^{\tensor k}\right)=0
\end{equation}
for all $p,q$ with $p+q=2$. For $p=2$ this vanishing is obvious for dimension reasons and for $p=1$ it follows from Serre's theorem if $k\gg 0$. 

For $p=0$ we claim that  $R^2 g_* \left(\pi^*(E\tensor L_0^{\tensor m}\tensor M)\right)=0$. By base change, see \cite[III. Corollary 3.5]{BS}, it suffices to prove that for every (!) fibre $F$ of $g$ we have 
\begin{equation}\label{eq: last computation}
\begin{split}
0 & =  H^2(F, \pi^*(E\tensor L_0^{\tensor m}\tensor M)\tensor g^* A^{\tensor k}|_F)\\ 
& =  H^2(F, \pi^*(E\tensor L_0^{\tensor m}\tensor M)|_F)\\
&=H^2(F, \pi^*E\vert_F \tensor M \tensor \sO_F(m\Sigma)) \\
&=H^0(F, \pi^*E^\vee\vert_F \tensor \omega_F \tensor M^\vee \tensor \sO_F(-m\Sigma))^\vee
\end{split}
\end{equation}
where in the second step we used that $\pi^*L_0\vert_F= \sO_Y(\Sigma)\vert_F$ thanks to \eqref{eq geradenbuendel} and for the last step we used Serre duality on the fibre. Note that $F$ is Cohen-Macaulay\footnote{It is worthwhile noting that even though Serre duality needs $F$ to be Cohen-Macaulay, we would have needed less as we only care about the duality between $H^0$ and $H^{\dim F}$, which holds for every compact complex space of pure dimension.} and $\omega_F = \omega_X(-F)|_F$ is locally free because $F \subset X$ is a Cartier divisor in the smooth manifold $X$.

On an irreducible fibre one can achieve the vanishing \eqref{eq: last computation} by choosing the appropriate value for $m$ by Lemma \ref{lem: vanishing by twisting}, because $\Sigma$ dominates $C$.
For reducible fibres the vanishing follows from \eqref{eq m vanishing} by an appropriate choice of $0 \ll n_d \ll \ldots \ll n_0 \ll k$. Note that as $k$ is the last parameter to be adjusted we can choose it in such a way that it simultaneously implies the vanishing \eqref{eq leray vanishing} needed above for $p=1$.

Therefore, for the appropriate choice of $m,k,n_0,\ldots,n_d$ there is no second cohomology along the fibres and thus $R^2 g_* \pi^*(E\tensor L_0^{\tensor m}\tensor M)=0$. This concludes the last step of the proof. 

\end{description}
\end{proof}
\begin{remark}\label{rem: gap}
We would like to explain, where the gap in \cite{CDP} is hidden and where the problem shows up in the proof given above. 

Lemma 1.5 in \cite{CDP} claims that if $g\colon Y\to C$ is a holomorphic map from a threefold to a curve and $E$ is a vector bundle on $Y$ then $R^ig_*E$ is locally free. 

As a counterexample consider an elliptic curve $E$ and on the product $E\times E$ the divisor $D = \{0\}\times E -\Delta$, where $\Delta$ is the diagonal. If $p$ is the projection to the second factor then $R^1p_*\ko_{E\times E}(D)$ is a skyscraper sheaf supported at $0$, in particular not locally free. Pulling back to $E\times E\times \IP^1$ we get the same phenomenon on threefold. 

If in the proof of Theorem \ref{theorem main2 new} in the last step $R^2 g_* \pi^*(E\tensor L_0^{\tensor m}\tensor M)$ would be locally free, then it would be enough to check the vanishing \eqref{eq: last computation} for the general fibre and there would be no need to treat also the reducible fibres. 
\end{remark}

\section{Examples of complex threefolds with \texorpdfstring{$b_2=0$}{b2=0}.}\label{sec examples}
In this section we give a simple construction to show that the class of manifolds satisfying the conditions of Theorem \ref{theorem main2 new} is non-empty. For special choices we recover some classical Calabi--Eckmann-manifolds \cite{ce53}, a construction which has by now been generalized in several directions, see for instance \cite{m00,b01, mv04,Catanese,  bo15,  puv16}.

Here we strive more for simplicity and thus do not work in the full possible generality.

\subsection{Holomorphic principal bundles with fibre an elliptic curve}

We fix an elliptic curve $ E = \IC/\Gamma$, that is, a compact complex Lie group of dimension $1$. We trivially get the short exact sequence of complex Lie groups
\begin{equation}\label{eq: ex seq E}
 \begin{tikzcd}
  0\rar&  \Gamma\rar & \IC\rar & E\rar & 0.
 \end{tikzcd}
\end{equation}
Now assume we have a holomorphic $E$-principal bundle $\pi\colon X \to B$ over a compact complex manifold $B$. Covering $B$ by open balls $U_i$ we know that $\inverse\pi (U_i)\isom U_i\times E$ as principal bundles over $U_i$, and we can therefore describe the principal bundle by a holomorphic cocycle with values in the structure group $E$. 

In other words, if 
\[ \ke_B \colon U \mapsto \{ \phi\colon U \to E \text{ holomorphic}\}\]
is the sheaf on $B$ of holomorphic maps to $E$, then $E$-principal bundles are classified by the \v{C}ech-cohomology group $\check H^1(B, \ke_B)$. 

To study this cohomology group, we use \eqref{eq: ex seq E}: the sheaf of  holomorphic maps to  $\Gamma$ is $\underline{\Gamma}_B$, the sheaf of locally constant functions with values in $\Gamma$, and the sheaf of holomorphic maps to $\IC$ is just the structure sheaf $\ko_B$, so that 
we get a short exact sequence of sheaves of abelian groups
\begin{equation}\label{eq: ex seq curly E}
 \begin{tikzcd}
  0\rar&  \underline{\Gamma}_B\rar & \ko_B\rar & \ke_B\rar & 0.
 \end{tikzcd}
\end{equation}
Possible holomorphic $E$-principal bundles over $B$ can now be studied via the associated long exact cohomology sequence and we do so in a special case. 
\begin{proposition}\label{prop: principal E bundles over B}
 Assume that $B$ is a compact complex  manifold with $h^1(B, \ko_B)=0$ and $E = \IC/\Gamma$ is an elliptic curve. Then:
 \begin{enumerate}
  \item Every holomorphic $E$-principal bundle $\pi\colon X\to B$ is classified (up to isomorphism) by its class $c(\pi)\in H^1(B, \Gamma)$.
  \item If in addition $h^2(\ko_B)=0$ then every class in $H^1(B, \Gamma)$ can be realized as the class of a holomorphic $E$-principal bundle. 
  \item If in addition $B$ is simply connected and $\pi\colon X\to B$ is a holomorphic $E$-principal bundle, then $\pi_i(X) = \pi_i(B)$ for $i\geq 3$ and there is an exact sequence 
  \[
   \begin{tikzcd}
    0\rar & \pi_2(X) \rar & H_2(B, \IZ)\rar{c(\pi)} & \Gamma\rar & \pi_1(X)\rar & 0,
   \end{tikzcd}
  \]
  where we consider $c(\pi)\in H^2(B, \Gamma) \isom \Hom(H_2(B, \IZ), \Gamma)$. 
 \end{enumerate}
\end{proposition}
\begin{proof}
The long exact cohomology sequence for \eqref{eq: ex seq curly E} reads 
 \begin{equation}\label{eq: LES}
   \begin{tikzcd}[column sep = small]
 H^1(B, \ko_B) \rar & H^1(B, \ke_B) \rar{c(-)} & H^2(B, \Gamma)\rar&  H^2(B, \ko_B)\rar & \dots
   \end{tikzcd},
  \end{equation}
which immediately yields the first two items under our assumptions. For the last item we use Hurewicz to identify $\pi_2(B)\isom H_2(B, \IZ)$ and the universal coefficient theorem to describe the items in the long exact sequence for the homotopy groups of the fibration $\pi$. The homotopy groups of $E$ are all trivial except $\pi_1(E) = \Gamma$, since the universal cover is contractible.
\end{proof}

\subsection{Examples}
Let $E = \IC/\Gamma$ be a fixed elliptic curve.

\begin{example}
 Let $B$ be a simply connected K\"ahler surface with $H_2(B, \IZ)\isom\IZ^2$ and $p_g (B) = h^2(B, \ko_B) = 0$. Choose an isomorphism $c\colon  H_2(B, \IZ)\isom \Gamma$.
 Then by Proposition \ref{prop: principal E bundles over B} there exist a principal bundle $\pi\colon X\to B$ with class $c(\pi)  = c$ and $X$ is simply connected with $H_2(X, \IZ)=0$.

 Choosing $B = \IP^1\times \IP^1$ recovers the construction of Calabi and Eckmann, taking $B$ to be one of the Hirzebruch surfaces falls in the larger class of examples covered by \cite[Sect.~5]{mv04}. 

 Note that in both these cases every meromorhic function in $\km (B) \isom \IC(x,y)$ pulls back to a meromorphic function on $X$ so that there are a lot of non-holomorphic, meromorphic functions on $X$. 
 
  In particular, $X$ satisfies the assumptions of Theorem \ref{theorem main2 new}.
 \end{example}

\begin{example}
 Let $\gamma$ be an element in a basis of $\Gamma$ and let $c \colon H_2(\IP^2, \IZ)\isom \IZ \to \Gamma$ be given by $1\mapsto \gamma$. 
 Considering the associated principal bundle $\pi\colon X\to \IP^2$ see that topologically $X = Y \times S^1$ where $Y$ is an $S^1$-bundle over $\IP^2$. By  Proposition \ref{prop: principal E bundles over B} and the K\"unneth formula we get an example where $\pi_1(X) = \IZ$ and $b_2(X)=0$. 
\end{example}

\begin{example}
 Let $B$ be a compact complex curve.  Let $\gamma$ be an element in a basis of $\Gamma$ and let $c \colon H_2(B, \IZ)= \IZ[B] \to \Gamma$ be given by $[B]\mapsto \gamma$. Since $H^2(B, \ko_B)=0$ for dimension reasons, we infer from the sequence \eqref{eq: LES} that $c$ considered as an element of $H^2(B, \Gamma)$ can be realised as the class of an $E$-principal bundle $\pi\colon X\to B$. Similar to the previous example we get $b_1(X) = b_1(B)+1 = 2g(B)+1$ by the Gysin sequence and the K\"unneth formula, so $X$ is a non-K\"ahler surface. 
 
For $B=\IP^1$ we get a special Hopf surface, compare Example \ref{ex: Hopf}. For an elliptic curve $B$ we get a Kodaira surface \cite[p. 197]{BHPV}. The latter is also known as Kodaira-Thurston manifold, because it was exhibited by Thurston in \cite{Thu} as the first example of a symplectic manifold that does not admit K\"ahler structures. 
\end{example}

\section{The automorphism group of a hypothetical complex \texorpdfstring{$\csphere$}{S6}}\label{sect: aut}
Another important and interesting object in the study of a compact complex manifold $X$ is its group of holomorphic automorphisms $\Aut_\sO(X)$. Recall that $\Aut_\sO(X)$ is a complex Lie group whose Lie algebra is the Lie algebra of holomorphic vector fields \cite{BM}.

In this section we discuss the result of Huckleberry, Kebekus, and Peternell on the automorphism group of a hypothetical complex structure on $\dsphere$. To state the theorem we need the following definition.

\begin{definition}
  A connected compact complex manifold $X$ is called almost homogeneous,
  if a closed complex subgroup $G$ of $\Aut_\sO(X)$ has an open orbit
  $\Omega$ in $X$.
  Then $\Omega$ is a dense open connected complex submanifold of $X$
  and $X \setminus \Omega$ is a proper analytic subset of $X$.
\end{definition}

{ 
We expect that the condition $a(\csphere) = 0$ can be removed\footnote{After the present article had been published, it was shown in \cite{CDP19} that $a(\csphere) = 0$ holds indeed.} in the following results, see Section \ref{sec main theorem}.
}

\begin{theorem}[Huckleberry, Kebekus, Peternell]\label{thm: almost_homogeneous}
 A complex $6$-sphere $\csphere$ with algebraic dimension  $a(\csphere)=0$ cannot be almost homogeneous.
\end{theorem}

We deduce a result in the spirit of the bounds on the Hodge-numbers explained in \cite{Angella}. 
\begin{corollary}
A complex $6$-sphere $\csphere$ with $a(\IS^6)=0$ carries at most $2$ linearly independent holomorphic vector fields, i.e.,
 $h^0(T_{\csphere}) \leq 2$.
\end{corollary}
\begin{proof}
  Note that the maximal dimension of an $\Aut_\ko(\csphere)$-orbit is 
  \[
	\max\left\{ k \,\middle|\, \Wedge^kH^0 (\csphere, T_{\csphere}) \to H^0 \left(\csphere, \Wedge^kT_{\csphere}\right) \text{ is non-zero}\right\}.
	\]
  Then the result follows directly from Theorem \ref{thm: almost_homogeneous} and Proposition 
  \ref{prop: wedging sections}. 
\end{proof}

\subsection{Ingredients of the proof}
The proof of Theorem \ref{thm: almost_homogeneous} is quite involved, but we don't feel we can improve on the presentation in \cite{HKP}; simply  expanding it to make it accessible to a wider audience would go beyond the scope of this paper. Let us however take a look what kind of techniques go into it.

We argue by contradiction: let us assume that $\csphere$ is a complex $6$-sphere with $a(\csphere) = 0$ which is quasi-homogeneous. Then Proposition \ref{prop: wedging sections} implies that $\dim \Aut_\ko(\csphere) =  h^0(\csphere, T_{\csphere})=3$. So passing to the identity component of the universal cover of the automorphism group, we are in the situation that a $3$-dimensional simply connected complex Lie group  $G$ with Lie algebra $H^0(\csphere, T_{\csphere}) = \C\langle X_1, X_2, X_3\rangle$ acts with discrete ineffectivity on $\csphere$ with an open orbit
$ \Omega = G\cdot x_0$
and complement $E = \csphere\setminus \Omega$ a compact complex surface 
defined by the section $X_1\wedge X_2\wedge X_3\in H^0\left(\csphere, \Wedge^3T_{\csphere}\right)$. Note that $E$ is not empty, because every vector field has a zero as $c_3(\csphere) = e(\csphere)\neq0$.  

Now the proof proceeds by playing off against each other 
\begin{itemize}
 \item the topology of $\csphere$, $\Omega$, and $E$;
 \item the classification of low-dimensional complex Lie groups and their actions;
 \item the analytic geometry of $\csphere$, e.g., the fact that $\csphere$ contains only finitely many hypersurfaces by our assumption $a(\csphere)=0$ and \cite{KRA};
 \item the known structure of compact complex surfaces to control the geometry of (a desingularization of) $E$;
\end{itemize}
and we refer to \cite{HKP} for the details.

\subsection{Connection to complex structures on $\IP^3$}
We will now explain a connection between the Hopf problem and the existence of exotic complex structures on complex projective space $\IP^3$.

First of all, denoting by $\bar \IP^3$ the differentiable manifold $\IP^3$ with the reversed orientation then complex conjugation and connected sum with a sphere induce diffeomorphisms $\IP^3\isom \bar\IP^3\isom \csphere\sharp \bar\IP^3$.

Now assume we have a complex $\csphere$ with $a(\csphere)=0$ and  denote the blow up at a point $p$ by $\pi_p\colon X_p\to \csphere$. Then $X_p$ is diffeomorphic to  the connected sum with $\bar\IP^3$, and hence to $\IP^3$ by the above. We can thus  consider the $X_p$ as a family of exotic\footnote{Since $a(X_p)=0$ the complex structure is clearly not the standard one.} complex structures on $\IP^3$, parametrized by $\csphere$.  Theorem~\ref{thm: almost_homogeneous} can be interpreted in this context as follows:

\begin{proposition}[Huckleberry, Kebekus, Peternell]
 Assume there exists a complex $\csphere$ with $a(\csphere)=0$. Then there is an at least $1$-dimensional non-trivial family of exotic complex structures on $\IP^3$. 
\end{proposition}
\begin{proof}
With the above notation let $E_p = \pi^{-1}(p)\isom \IP^2$ be the exceptional divisor of the blow up. 

Assume $\psi\colon X_p \to X_q$ is a biholomorphic map.
  The image $\psi(E_p)$ generates $H^2(X_q,\IZ)= \IZ[E_q]$ and $[E_q]^2\neq 0 $, so its intersection
  with $E_q$ is a curve or $E_q$ itself.
If $\psi(E_p)\cap E_q$ is a curve then the composition $\pi_q\circ \psi|_{E_p}$ is a non-constant map from $\IP^2$ to a surface which contracts a curve to a point. Such a map does not exist. 
    So $\psi(E_p) = E_q$ and $\psi$ covers a holomorphic automorphism $\bar \psi\in \Aut_\ko(\csphere)$.
  We conclude that $X_p$ and $X_q$ are biholomorphic if and only if $p$ and $q$ are in the same   $\Aut_\ko(\csphere)$-orbit. 
    
  By Theorem \ref{thm: almost_homogeneous} the complex manifold  $\csphere$ is not almost homogeneous and therefore we can choose a small $1$-dimensional embedded disc $\Delta$ which is transversal to the orbits of the automorphism group. Then $\{X_p\}_{p\in \Delta}$ gives the desired $1$-dimensional family of exotic complex structures on $\IP^3$. 
\end{proof}
The proof shows that in a set-theoretic sense the orbit space of the $\Aut_\ko(\IS^6)$-action classifies the exotic complex structures that we obtain, but this orbit space will usually not be a complex analytic space.



\begin{thebibliography}{BCHM10}



\bibitem[AK73]{AK}
Aldo Andreotti and Arnold Kas.
\newblock Duality on complex spaces,
\newblock \emph{Ann. Scuola Norm. Sup. Pisa} (3) 27 (1973), 187--263.

\bibitem[An17]{Angella}
Daniele Angella.
\newblock Hodge numbers of a hypothetical complex structure on $S^6$,
\newblock \emph{Differential Geom. Appl.} 57 (2018), 105-120.

\bibitem[BHPV04]{BHPV}
Wolf P. Barth, Klaus Hulek, Chris A. M. Peters, and Antonius Van de Ven.
\newblock \emph{Compact complex surfaces},
\newblock Second edition. Ergebnisse der Mathematik und ihrer Grenzgebiete. 3. Folge. A Series of Modern Surveys in Mathematics, 4. Springer-Verlag, Berlin, 2004. xii+436 pp.

\bibitem[BO15]{bo15}
Laurent Battisti and Karl Oeljeklaus.
\newblock A generalization of {S}ankaran and {LVMB} manifolds,
\newblock {\em Michigan Math. J.}, 64(1), 203--222, 2015.

\bibitem[BFMP79]{BFMP2}
Paul Baum, William Fulton, and Robert MacPherson.
\newblock Riemann-Roch and topological K-theory for singular varieties,
\newblock \emph{Acta Math.} 143 (1979), no. 3--4, 155--192. 

\bibitem[BM47]{BM}
Salomon Bochner and Deane Montgomery. Groups on analytic manifolds,
\emph{Ann. of Math. (2) 48},  659--669, 1947 	 	. 

\bibitem[Bo01]{b01}
Fr\'ed\'eric Bosio.
\newblock Vari\'et\'es complexes compactes: une g\'en\'eralisation de la
  construction de {M}eersseman et {L}\'opez de {M}edrano-{V}erjovsky,
\newblock {\em Ann. Inst. Fourier (Grenoble)}, 51(5), 1259--1297, 2001.

\bibitem[BS77]{BS}
Constantin B{\u{a}}nic{\u{a}} and Octavian St{\u{a}}n{\u{a}}{\c{s}}il{\u{a}}.
\newblock \emph{M\'ethodes alg\'ebriques dans la th\'eorie globale des espaces
  complexes. {V}ol. 1},
\newblock Paris : Gauthier-Villars, 1977. --
\newblock Avec une pr{\'e}face de Henri Cartan, Troisi{\`e}me {\'e}dition,
  Collection ``Varia Mathematica''.
	
\bibitem[CE53]{ce53}
Eugenio Calabi and Beno Eckmann.
\newblock A class of compact, complex manifolds which are not algebraic,
\newblock {\em Ann. of Math. (2)}, 58, 494--500, 1953.

\bibitem[CDP98]{CDP}
Fr\'ed\'eric Campana, Jean-Pierre Demailly, and Thomas Peternell.
\newblock The algebraic dimension of compact complex threefolds with vanishing second Betti number,
\newblock \emph{Compositio Math.} 112 (1998), 77--91.

\bibitem[CDP19]{CDP19}
Fr\'ed\'eric Campana, Jean-Pierre Demailly, and Thomas Peternell.
\newblock The algebraic dimension of compact complex threefolds with vanishing second Betti numbers,
\newblock preprint {\tt arXiv:1904.11179}, 2019.

\bibitem[Ca04]{Catanese}
Fabrizio Catanese.
\newblock Deformation in the large of some complex manifolds. I,
\newblock \emph{Ann. Mat. Pura Appl.} (4) 183 (2004), no. 3, 261--289.

\bibitem[Fi76]{Fi}
Gerd Fischer.
\newblock \emph{Complex analytic geometry},
\newblock Lecture Notes in Mathematics, Vol. 538. Springer-Verlag, Berlin-New York, 1976. vii+201 pp.

\bibitem[Fu98]{fulton}
William Fulton. 
\newblock \emph{Intersection theory}, Second edition. Ergebnisse der Mathematik und ihrer Grenzgebiete. 
\newblock 3. Folge. A Series of Modern Surveys in Mathematics, 
\newblock Springer-Verlag, Berlin, 1998. xiv+470 pp. ISBN: 3-540-62046-X; 0-387-98549-2.

\bibitem[Gr60]{grauert}
Hans Grauert.
\newblock Ein Theorem der analytischen Garbentheorie und die Modulr\"aume komplexer Strukturen,
\newblock {\em  Publ. Math. IH\'ES}, Volume 5 (1960) , p. 5-64.

\bibitem[GPR94]{SCV7}
Hans Grauert, Thomas Peternell, and Reinhold Remmert, editors.
\newblock {\em Several complex variables {VII}, Sheaf-theoretical methods in
  complex analysis}, volume~74 of {\em Encyclopaedia of Mathematical Sciences}.
\newblock Springer-Verlag, Berlin, 1994.

\bibitem[GR84]{GR}
Hans Grauert and Reinhold Remmert.
\newblock \emph{Coherent analytic sheaves},
\newblock Grundlehren der Mathematischen Wissenschaften, 265. Springer-Verlag, Berlin, 1984. xviii+249 pp. ISBN: 3-540-13178-7

\bibitem[Ha77]{Ha}
Robin Hartshorne.
\newblock \emph{Algebraic geometry}, Graduate Texts in Mathematics, No. 52. 
\newblock Springer-Verlag, New York-Heidelberg, 1977. xvi+496 pp. ISBN: 0-387-90244-9.

\bibitem[HKP00]{HKP}
Alan T. Huckleberry, Stefan Kebekus and Thomas Peternell.
\newblock Group actions on $S^6$ and complex structures on $\IP^3$,
\newblock \emph{Duke Math. J.} 102 (2000), no. 1, 101--124.

\bibitem[Hu04]{Hu}
Daniel Huybrechts.
\newblock \emph{Complex geometry - an introduction},
\newblock Springer (2004). Universitext. 309 pages.

\bibitem[Ko66]{kodaira123}
Kunihiko Kodaira.
\newblock  On the structure of compact complex analytic surfaces II,
\newblock \emph{Amer. J. Math.} 88 1966 682--721.

\bibitem[Ko87]{kodaira}
Kunihiko Kodaira.
\newblock  Complex structures on $\IS^1\times\IS^3$,
\newblock \emph{Proc. Nat. Acad. Sci. U.S.A.} 55, 1966, 240--243.

\bibitem[Kr86]{KRA}
V. A. Krasnov.
\newblock Compact complex manifolds without meromorphic functions (in Russian),
\newblock \emph{Mat. Zametki} 17 (1975), 119--122; English translation in \emph{Math. Notes} 17 (1975),
69--71.

\bibitem[La11]{lattes}
Samuel Latt\`es.
\newblock Sur les formes r\'eduites des transformations ponctuelles dans le domaine d'un point double,
\newblock \emph{Bull. Soc. Math. France} 39 (1911), 309--345.

\bibitem[Le87]{L}
Roni N. Levy.
\newblock  The Riemann-Roch theorem for complex spaces,
\newblock \emph{Acta Math.} 158 (1987), no. 3--4, 149--188.

\bibitem[Me00]{m00}
Laurent Meersseman.
\newblock A new geometric construction of compact complex manifolds in any
  dimension,
\newblock {\em Math. Ann.}, 317(1), 79--115, 2000.

\bibitem[MV04]{mv04}
Laurent Meersseman and Alberto Verjovsky.
\newblock Holomorphic principal bundles over projective toric varieties,
\newblock {\em J. Reine Angew. Math.}, 572, 57--96, 2004.

\bibitem[OBTT81]{OBTT}
Nigel R. O'Brian, Domingo Toledo, and Yue Lin L. Tong.
\newblock  The Trace Map and Characteristic Classes for Coherent Sheaves,
\newblock {\em American Journal of Mathematics} Vol. 103, No. 2 (Apr., 1981), pp. 225--252. 

\bibitem[PUV16]{puv16}
Taras Panov, Yury Ustinovskiy, and Misha Verbitsky.
\newblock Complex geometry of moment-angle manifolds,
\newblock {\em Math. Z.}, 284(1-2), 309--333, 2016.

\bibitem[Pa17]{maurizio}
Maurizio Parton.
\newblock Almost complex structures on spheres,
\newblock \emph{Differential Geom. Appl.} 57 (2018), 10-22.

\bibitem[Si55]{Siegel}
Carl Ludwig Siegel.
\newblock  Meromorphe Funktionen auf kompakten analytischen Mannigfaltigkeiten,
\newblock \emph{Nachr. Akad. Wiss. G\"ottingen. Math.-Phys. Kl. IIa.} 1955, 71--77.

\bibitem[St57]{sternberg}
Shlomo Sternberg.
\newblock Local contractions and a theorem of Poincar\'e,
\newblock \emph{Amer. J. Math.} 79, 1957, 809--824.

\bibitem[Th54]{Th}
Walter Thimm.
\newblock \"Uber meromorphe Abbildungen von komplexen Mannigfaltigkeiten,
\newblock \emph{Math. Ann.} 128, (1954). 1--48.

\bibitem[Th66]{Th2}
Walter Thimm.
\newblock Der Weierstra\ss sche Satz der algebraischen Abh\"angigkeit von Abelschen Funktionen und seine Verallgemeinerungen,
\newblock \emph{Festschr. Ged\"achtnisfeier K. Weierstrass} pp. 123--154 Westdeutscher Verlag, Cologne 1966.

\bibitem[Th76]{Thu}
William Paul Thurston.
\newblock Some simple examples of symplectic manifolds,
\newblock \emph{Proc. Amer. Math. Soc.} 55 (1976), no. 2, 467--468.

\end{thebibliography}
\end{document}